\documentclass[smallextended]{svjour3}       
\smartqed  
\usepackage[english]{babel}
\usepackage[utf8]{inputenc}
\usepackage{amssymb, mathtools, enumerate, amsmath}
\usepackage{tabularx,lipsum,environ, caption}
\usepackage{stmaryrd}
\usepackage{xspace}

\newcommand{\Nat}{{\mathbb N}}

\newcommand{\TII}{\xrightarrow{ii}}
\newcommand{\Dom}{\xrightarrow{d}}
\newcommand{\bhom}{\xrightarrow{b}}
\newcommand{\Hom}{\rightarrow}
\newcommand{\fa}{f_1}
\newcommand{\fb}{f_2}

\begin{document}

\title{Fall-colorings and b-colorings of graph products
\thanks{This work was partially supported by CNPq, Brazil.}
}

\author{Ana Silva}


\institute{ParGO Research Group - Parallelism, Graphs and Optimization. 
            Departamento de Matem\'atica, Universidade Federal do Cear\'a, Brazil. 
            \email{ana.silva@mat.ufc.br}
}

\date{Received: date / Accepted: date}

\maketitle

\begin{abstract}
Given a proper coloring $f$ of a graph $G$, a b-vertex in $f$ is a vertex that is adjacent to every color class but its own. It is a b-coloring if every color class contains at least one b-vertex, and it is a fall-coloring if every vertex is a b-vertex. The b-chromatic number of $G$ is the maximum integer $b(G)$ for which $G$ has a b-coloring with $b(G)$ colors, while the fall-chromatic number and the fall-acromatic number of $G$ are, respectively, the minimum and maximum integers $\fa(G),\fb(G)$ for which $G$ has a fall-coloring. In this article, we explore the concepts of b-homomorphisms and Type II homomorphisms, which generalize the concepts of b-colorings and fall-colorings, and present some meta-theorems concerning products of graphs. As a result, we derive some previously known facts about these metrics on graph products. We also give a negative answer to a question posed by Kaul and Mitillos about fall-colorings of perfect graphs.

\end{abstract}
 
\section{Introduction}\label{intro}

Given a simple graph $G$\footnote{The graph terminology used in this paper follows \cite{BM08}.}, and a function $f: V(G)\rightarrow\{1, \cdots, k\}$, we say that $f$ is a \emph{proper coloring of $G$ with $k$ colors} if $f(u) \neq f(v)$ for every $uv \in E(G)$. The \emph{chromatic number of $G$} is the minimum value $k$ for which $G$ has a proper coloring with $k$ colors; it is denoted by $\chi(G)$. The related decision problem is one of the Karp's~21 NP-complete problems~\cite{K.72}, and it continues to be NP-complete even if $k$ is considered to be fixed~\cite{Hol81}. The chromatic number is also hard to approximate: for all $\epsilon > 0$, there is no algorithm that approximates the chromatic number within a factor of $n^{1 - \epsilon}$ unless P = NP~\cite{Has96,Zuc07}.

The graph coloring problem and its variants are perhaps the most studied problems in graph theory, in part due to its wide range of applications in practice.
For instance, problems of \emph{scheduling}~\cite{Werra.85}, \emph{frequency assignment}~\cite{Gamst.86}, \emph{register allocation}~\cite{Chow.Hennessy.84,Chow.Hennessy.90}, and the \emph{finite element method}~\cite{Saad.96}, are naturally modelled by colourings.

Given its difficulty, one approach to obtain proper colorings of a graph is to use coloring heuristics.
Consider a proper coloring $f$ of graph $G$ that uses $k$ colors. A value $i$ in $\{1,\cdots,k\}$ is called \emph{color $i$}, and the set of vertices $f^{-1}(i)$ is called \emph{color class $i$}. 
A vertex $v$ in color class $i$ is called a \emph{b-vertex of color $i$} if $v$ has at least one neighbor in color class $j$, for every $j\in\{1,\cdots,k\}$, $j\neq i$.
If color $i$ has no $b$-vertices, we may recolor each $v$ in color class $i$ with some color that does not appear in the neighborhood of $v$.
In this way, we eliminate color $i$, and obtain a new proper coloring of $G$ that uses $k - 1$ colors.
The procedure may be repeated until we reach a coloring such that every color class contains a $b$-vertex.
Such a coloring is called a \emph{$b$-coloring}. Clearly, if $k=\chi(G)$, then the described procedure cannot decrease the number of colors used in $f$. This means that every optimal coloring of $G$ is also a b-coloring and this is why we are only interested in investigating the worst-case scenario for the described procedure. 
The \emph{$b$-chromatic number} of a graph $G$, denoted  by $b(G)$, is the largest $k$ such that $G$ has a $b$-coloring with $k$ colors.
This concept was introduced by Irving and Manlove in~\cite{IM99}, where they prove that  determining the $b$-chromatic number of a graph is an NP-complete problem. In fact, it remains so even when restricted to bipartite graphs~\cite{KTV.02}, connected chordal graphs~\cite{HLS11}, and line graphs~\cite{CLMSSS15}.

A related type of coloring is the fall-coloring. A proper coloring $f$ of $G$ is called a \emph{fall-coloring of $G$} if every vertex of $G$ is a b-vertex in $f$. Unlike the b-colorings, some graph may not have a fall-coloring. For instance, if $\delta(G)$ denotes the minimum degree of a vertex in $G$ and $\chi(G) > \delta(G)+1$, then no vertex with minimum degree can be a b-vertex; hence $G$ does not have a fall-coloring. Also, even if $G$ does admit a fall-coloring, it is not necessarily true that it admits a fall-coloring with $\chi(G)$ colors. Therefore, we define the \emph{fall-spectrum} of $G$ as being the set ${\cal F}(G)$ containing every $k$ for which $G$ admits a fall-coloring with $k$ colors. If ${\cal F}(G)\neq \emptyset$, then the \emph{fall-chromatic number of $G$} is the minimum value $\fa(G)$ in ${\cal F}(G)$, while the \emph{fall-achromatic number of $G$} is the maximum value $\fb(G)$ in ${\cal F}(G)$. This concept was introduced in~\cite{DHHJKLR.00}, where they also show that deciding whether ${\cal F}(G)\neq \emptyset$ is NP-complete. We mention that some authors have used $\chi_f(G),\psi_f(G)$ to denote $\fa(G),\fb(G)$, respectively, which we do not adopt here since $\chi_f(G)$ is more largely used to denote the fractional chromatic number of $G$.  Observe that, if ${\cal F}(G)\neq \emptyset$, then:
\[\chi(G)\le \fa(G) \le \fb(G)\le \delta(G)+1\]

A concept related to b-colorings that is analogous to the fall-spectrum is that of the b-spectrum. In~\cite{KTV.02} it is proved that $K'_{p, p}$, the graph obtained from $K_{p,p}$ by removing a perfect matching, admits $b$-colorings only with $2$ or $p$ colors.
And in~\cite{BCF07}, the authors prove that, for every finite $S\subset \Nat-\{1\}$, there exists a graph $G$ that admits a b-coloring with $k$ colors if and only if $k\in S$.
Motivated by these facts, in~\cite{BCF03} the authors define the \emph{b-spectrum} of a graph $G$ as the set containing every positive value $k$ for which $G$ admits a b-coloring with $k$ colors; this is denoted by $S_b(G)$. Also, they say that $G$ is \emph{b-continuous} if $S_b(G)$ contains every integer in the closed interval $[\chi(G),b(G)]$. 

It is well known that graph homomorphisms generalize proper colorings. Given graphs $G$ and $H$, a function $f:V(G)\rightarrow V(H)$ is a \emph{homomorphism} if every edge of $G$ is mapped into an edge of $H$, i.e., if $f(u)f(v)\in E(H)$, for every $uv\in E(G)$. If such a function exists, we write $G\Hom H$. One can easily verify that $G\Hom K_n$ if and only if $\chi(G)\le n$. In fact, this is a very rich subject that has been largely studied. We direct the interested reader to~\cite{HN.book}.

Recently, special types of homomorphisms that generalize b-colorings and fall-colorings have also been independently used in the study of the b-continuity of graphs and of certain products of graphs~\cite{LL.09,SLS.16,S.15}. 
In~\cite{LL.09}, the authors prove that the existence of a Type II homomorphism, which generalizes fall-colorings, is a transitive relation, and use the concept to investigate the fall-colorings of the cartesian products of graphs. 
Similarly, in~\cite{S.15}, the author prove that the existence of a semi-locally surjective homomorphism, which generalizes b-colorings, is a transitive relation and use the concept to prove the b-continuity of certain Kneser graphs. 
We mention that semi-locally surjective homomorphisms were studied independently in~\cite{SLS.16}, where the the concept is used to investigate the b-colorings of the lexicographic products of graphs; there, the authors use the term b-homomorphisms, which we give preference because of its brevity. 

In this article, we show that these results can actually be produced for the main existing products of graphs. For this, we generalize the concept of a graph product and present our results in the form of meta-theorems. In particular, the theorems below follow directly from these meta-theorems and some easy observations regarding these products, which we will see in Section~\ref{sec:products}. There, the reader can find Table~\ref{tab:products}, which contains the formal definition of each of the products in the theorems below. We mention that, in addition to generalizing results presented in the previously cited articles, our theorems also generalize results presented in~\cite{JP.15,KM.02,KTV.02,S.15_2}. Although our proof need some heavy notation, it has the advantage of proving all of these results at once.

\begin{theorem}\label{theo:mainb}
Let $G,H$ be graphs, and $\odot$ denote a graph product. Then, 
\begin{itemize}
\item If $\odot$ is either the lexicographic product, or the strong product, or the co-normal product, then
\[b(G\odot H) \ge b(G)b(H);\]
\item And if $\odot$ is the cartesian product or the direct product, then
\[b(G\odot H)\ge \max\{b(G), b(H)\}.\]
\end{itemize}
\end{theorem}

\begin{theorem}\label{theo:mainfall}
Let $G,H$ be graphs, and $\odot$ denote a graph product. If ${\cal F}(G)\neq \emptyset$ and ${\cal F}(H)\neq \emptyset$, then ${\cal F}(G\odot H)\neq \emptyset$. Also,
\begin{itemize}
\item If $\odot$ is either the lexicographic product, or the strong product, or the co-normal product, then
\[\fa(G\odot H)\le \fa(G)\fa(H)\le \fb(G)\fb(H)\le \fb(G\odot H);\]
\item If $\odot$ is the cartesian product, then
\[\fa(G\odot H)\le \max\{\fa(G),\fa(H)\}\le \max\{\fb(G),\fb(H)\}\le \fb(G\odot H);\]
\item And if $\odot$ is the direct product, then
\[\fa(G\odot H)\le \min\{\fa(G),\fa(H)\}\le \max\{\fb(G),\fb(H)\}\le \fb(G\odot H).\]
\end{itemize}
\end{theorem}

We mention that our results also give information about the b-spectrum and fall-spectrum of the products.




Our article is organized as follows. In Section~\ref{sec:hom}, we present the main definitions and the results concerning b-homomorphisms and Type II homomorphisms of products of graphs. In Section~\ref{sec:products}, we present the formal definition of the main graph products, analyse the structure of the products of complete graphs, and present bounds for the metrics on these products. The results on these two sections produce Theorems~\ref{theo:mainb} and~\ref{theo:mainfall}. In Section~\ref{sec:fallPart} we present some cases where ${\cal F}(G\odot H)$ can be non-empty even though ${\cal F}(H)$ is empty. 
 Finally, in Section~\ref{sec:conclusion} we present some questions left open, and show an example that give a negative answer to a question posed by Kaul and Mitillos about fall-colorings of perfect graphs~\cite{KM.16}.

\section{Homomorphisms and products}\label{sec:hom}

Given graphs $G$ and $H$, a \emph{graph product} $\odot$ on $G$ and $H$ is a graph $F$ such that $V(F)=V(G)\times V(H)$, and $\alpha\beta\in E(F)$ if and only if some condition $P_\odot(G,H,\alpha,\beta)$ is satisfied.
Given a vertex $u\in V(G)$, we denote by $V(u,H)$ the subset $\{(u,v)\mid v\in V(H)\}$, and the \emph{fiber of $u$ in $G\odot H$} is the subgraph of $G\odot H$ induced by $V(u,H)$. Given $v\in V(H)$, the subset $V(v,G)$ and fiber of $v$ are defined similarly. If there is no ambiguity, we ommit $H$ and $G$ in $V(u,H),V(v,G)$, respectively.

    We say that $\odot$ is an \emph{adjacency product} if $P_\odot(G,H,\alpha,\beta)$ is a composition of a subset of the following formulas, where $\alpha=(u_a,v_a)$ and $\beta=(u_b,v_b)$: 
\[{\cal B}(G,H,\alpha,\beta) = \{u_a=u_b,\ v_a=v_b,\ u_au_b\in E(G),\ v_av_b\in E(H)\}.\]

These are called \emph{basic formulas related to $(\alpha,\beta)$}, where $\alpha,\beta\in V(G)\times V(H)$. In Section~\ref{sec:products}, we present the formal definitions of the main studied adjacency products. The next proposition will be very useful throughout this section.

\begin{proposition}\label{prop:basic}
Let $G,H,G',H'$ be graphs, $\odot$ be an adjacency product, and consider vertices $\alpha,\beta \in V(G\odot H)$, and $\alpha'\beta'\in V(G'\odot H')$. If for every basic formula $\gamma(G,H,\alpha,\beta)$ in ${\cal B}(G,H,\alpha,\beta)$ we have that $\gamma(G,H,\alpha,\beta)$ implies $\gamma(G',H',\alpha',\beta')$, then
 \[P_\odot(G,H,\alpha,\beta) \Rightarrow P_\odot(G',H',\alpha',\beta').\]
\end{proposition}

The next lemma tell us that graph homomorphisms are well behaved under adjacency products.

\begin{lemma}\label{lem:hom}
Let $G$, $H$ and $F$ be graphs and $\odot$ be an adjacency product. If $H\Hom F$, then $(G\odot H)\Hom (G\odot F)$ and $(H\odot G)\Hom (F\odot G)$.
\end{lemma}
\begin{proof}
Let $f$ be a homomorphism from $H$ to $F$, and denote $(G\odot H)$ and $(G\odot F)$ by $H',F'$, respectively. We prove that $H'\Hom F'$, and the other part of the lemma is analogous. For this, let $g:V(H')\Hom V(F')$ be defined as $g((u,v)) = (u,f(v))$. Let $\alpha\beta\in E(H')$; we need to prove that $g(\alpha)g(\beta)\in E(F')$. 

Write $\alpha$ and $\beta$ as $(u_a,v_a)$ and $(u_b,v_b)$, respectively, and let $\alpha'=g(\alpha) = (u_a,f(v_a))$ and $\beta' = g(\beta) = (u_b,f(v_b))$. Recall that:
\[{\cal B}(G,H,\alpha,\beta) = \{u_a=u_b,v_a=v_b, u_au_b\in E(G), v_av_b\in E(H)\}\mbox{, and}\]
\[{\cal B}(G,F,\alpha',\beta') = \{u_a=u_b,f(v_a)=f(v_b), u_au_b\in E(G), f(v_a)f(v_b)\in E(F)\}.\]
Clearly $v_a=v_b$ implies $f(v_a)=f(v_b)$, and since $f$ is a homomorphism we know that $v_av_b\in E(H)$ implies $f(v_a)f(v_b)\in E(F)$. The lemma follows by Proposition~\ref{prop:basic} and the fact that $\alpha\beta\in E(G\odot H)$, i.e., $P_\odot(G,H,\alpha,\beta)$ holds. 
\end{proof}

In the following subsections, we formally define and analyse analogous properties concerning b-homomorphisms and Type II homomorphisms.



\subsection{b-homomorphism}

Given graphs $G$ and $H$, and a function $f:V(G)\Hom V(H)$, we say that $f$ is a b-homomorphism if $f$ is a homomorphism and for every $u\in V(H)$, there exists $u'\in f^{-1}(u)$ such that $f(N_G(u')) = N_H(u)$, where $f(X)$ denotes $\{f(x)\mid x\in X\}$. If such a function exists, we write $G\bhom H$. Observe that $f$ is always a surjective function. 
The following is an important property of b-homomorphism.

\begin{proposition}[\cite{SLS.16}]
If $G\bhom H$ and $H\bhom F$, then $G\bhom F$.
\end{proposition}

The following lemma is analogous to Lemma~\ref{lem:hom} and have been proved in~\cite{SLS.16} for the lexicographic product. 

\begin{lemma}\label{lem:bhom}
Let $G$, $H$, and $F$ be graphs and $\odot$ be an adjacency product. If $H\bhom F$, then $(G\odot H)\bhom (G\odot F)$ and $(H\odot G)\bhom (F\odot G)$.
\end{lemma}
\begin{proof}
Let $f$ be a b-homomorphism from $H$ to $F$, and denote $G\odot H$ and $G\odot F$ by $H',F'$, respectively. Define $g:V(H')\rightarrow V(F')$ as $g((u,v)) = (u,f(v))$. We prove that $g$ is a b-homomorphism; the other part of the theorem is analogous.

  By Lemma~\ref{lem:hom}, we know that $g$ is a homomorphism. So now consider $\alpha=(u_a,v_a)\in V(F')$; we need to show that there exists $\alpha'=(u'_a,v'_a)\in g^{-1}(\alpha)$ such that $g(N_{H'}(\alpha')) = N_{F'}(\alpha)$. Because $f$ is a b-homomorphism, there exists $v'_a\in f^{-1}(v_a)$ such that $f(N_H(v'_a)) = N_{\cal F}(v_a)$. So let $\alpha' = (u_a,v'_a)$, and consider any $\beta = (u_b,v_b)\in N_{F'}(\alpha)$. 
  We want to prove that there exists $\beta'\in N_{H'}(\alpha')$ such that $g(\beta')=\beta$. Recall that:
\[{\cal B}(G,F,\alpha,\beta) = \{u_a=u_b,v_a=v_b, u_au_b\in E(G), v_av_b\in E(F)\}\]

And for any $\beta' = (u_b,v'_b)$ where $v'_b\in E(H)$, we have: 
\[{\cal B}(G,H,\alpha',\beta') = \{u_a=u_b,v'_a=v'_b, u_au_b\in E(G), v'_av'_b\in E(F)\}\]
  
We want to find $v'_b$ such that $g(\beta') = \beta$ and that makes the basic formulas in the first equation imply the basic formulas in the second. If this is the case, then Proposition~\ref{prop:basic} implies $\beta'\in N_{F'}(\alpha')$ and we are done. 
If $v_a=v_b$, then let $v'_b$ be $v'_a$.
If $v_av_b\in E(F)$, then choose any $v'_b\in N_H(v'_a)$ such that $f(v'_b) = v_b$ (it exists by the choice of $v'_a$). Finally, if $v_a\neq v_b$ and $v_av_b\notin E(F)$, just let $v'_b$ be any vertex in $f^{-1}(v_b)$ (it exists since $f$ is surjective). One can verify that $v'_b$ is the desired vertex.
\end{proof}

\begin{corollary}
Let $G,H,G',H'$ be graphs and $\odot$ be an adjacency product. If $G\bhom G'$ and $H\bhom H'$, then $(G\odot H)\bhom (G'\odot H')$.
\end{corollary}

\begin{corollary}\label{cor:bspectrum}
Let $G,H$ be graphs and $\odot$ be an adjacency product. Then:
\[\bigcup_{\substack{k\in S_b(G)\\k'\in S_b(H)}}S_b(K_k\odot K_{k'})\subseteq \left(\bigcup_{k\in S_b(H)}S_b(G\odot K_k)\right) \cap \left(\bigcup_{k\in S_b(G)}S_b(K_k\odot H)\right), \]
and
\[\left(\bigcup_{k\in S_b(H)}S_b(G\odot K_k)\right) \cup \left(\bigcup_{k\in S_b(G)}S_b(K_k\odot H)\right)\subseteq S_b(G\odot H).\]
\end{corollary}

It is known that, contrary to the chromatic number, the b-chromatic number is not a monotonic parameter, i.e., a graph $G$ might have a subgraph $H$ such that $b(H) > b(G)$. For instance, let $H$ be obtained from the complete bipartite graph $K_{3,3}$ by removing a perfect matching, and let $G$ be obtained from $H$ by adding vertices $u,v$, edge $uv$ and making $u$ complete to one of the parts of $H$ and $v$ to the other. One can verify that $b(G) = 2 < b(H) = 3$. This is why we cannot ensure that the maximum value in the sets of Corollary \ref{cor:bspectrum} are attained when $k=b(G)$ and $k'=b(H)$. Nevertheless, we get:

\begin{corollary}\label{lem:bGdotH}
Let $G$ and $H$ be graphs, and $\odot$ be an adjacency produt. Also, let $S$ denote the set $\bigcup_{k\in S_b(G),k'\in S_b(H)}b(K_k\odot K_{k'})$. Then: 
\[b(G\odot H)\ge \max\{k\mid k\in S\}\ge b(K_{b(G)}\odot K_{b(H)}).\]
\end{corollary}


\subsection{Type II homomorphism}

Given graphs $G$ and $H$, a function $f:V(G)\rightarrow V(H)$ is a \emph{domatic homomorphism} if for every $u'\in V(H)$, $v'\in N(u')$, and $u\in f^{-1}(u')$, there exists $v\in f^{-1}(v)$ such that $uv\in E(G)$. Observe that a domatic homomorphism is not necessarily a homomorphism. In~\cite{LL.09}, the authors define a \emph{Type II homomorphism} as being a homomorphism which is also a domatic homomorphism. They then prove that the existence of such an homomorphism is a transitive relation, and investigate the cartesian product of graphs, in particular of two trees. If there exists a domatic homomorphism or Type II homomorphism from $G$ to $H$, then we write $G\Dom H$ or $G\TII H$, respectivelly. 
We want to prove an analogous version of Lemma~\ref{lem:bhom} for Type II homomorphisms. Because of Lemma~\ref{lem:hom}, we only need to prove the following.


\begin{lemma}\label{lem:dom}
Let $G$, $H$, and $F$ be graphs and $\odot$ be an adjacency product. If there exists a surjective domatic homomorphism $f$ from $H$ to $F$, then $(G\odot H)\Dom (G\odot F)$ and $(H\odot G)\Dom (F\odot G)$.
\end{lemma}
\begin{proof}
Again, denote $G\odot H, G\odot F$ by $H',F'$, respectively, and define $g:V(H')\rightarrow V(F')$ as $g((u,v)) = (u,f(v))$. Let $\alpha\beta\in E(F')$ and $\alpha'\in g^{-1}(\alpha)$. We want to prove that there exists $\beta'\in g^{-1}(\beta)$ such that $\alpha'\beta'\in E(H')$. Write $\alpha,\alpha',\beta$ as $(u_a,v_a),(u_a,v'_a),(u_b,v_b)$, respectively. 
Recall that:
\begin{equation}{\cal B}(G,F,\alpha,\beta) = \{u_a=u_b,v_a=v_b, u_au_b\in E(G), v_av_b\in E(F)\}.\label{eq1}\end{equation}

And for any $\beta' = (u_b,v'_b)$ where $v'_b\in E(H)$, we have: 
\begin{equation}{\cal B}(G,H,\alpha',\beta') = \{u_a=u_b,v'_a=v'_b, u_au_b\in E(G), v'_av'_b\in E(F)\}\label{eq2}\end{equation}

Again, we want to choose $v'_b$ such that the following holds:\\
\begin{itemize}
\item[(*)] $g(\beta') = \beta$, and the basic formulas in~(\ref{eq1}) imply the ones in~(\ref{eq2}).
\end{itemize}

If $v_a=v_b$, again let $v'_b$ be $v'_a$. We get $g(\beta') = (u_b,f(v'_a)) = (u_b,v_a) = \beta$. Also, because $F$ is a graph, we get that $v_av_a\notin E(F)$, in which case (*) can be verified. 
If $v_av_b\in E(F)$, choose $v'_b\in f^{-1}(v_b)$ such that $v'_av'_b\in E(H)$ (it exists because $f$ is a domatic homomorphism and $v'_a\in f^{-1}(v_a)$). We get $g(\beta') = (u_b,v_b) = \beta$, and because $v_av_b\in E(F)$ implies that $v_a\neq v_b$, one can verify that (*) holds.  
Similarly, if $v_a\neq v_b$ and $v_av_b\notin E(F)$, just let $v'_b$ be any vertex in $f^{-1}(v_b)$ (it exists since $f$ is surjective). 
\end{proof}

Observe that if $H$ does not have any isolated vertices, then every domatic homomorphism into $H$ is also surjective. Therefore, we get the following corollary.

\begin{corollary}
Let $G,H,G',H'$ be graphs and $\odot$ be an adjacency product. If $G'$ and $H'$ have no isolated vertices, $G\TII G'$ and $H\TII H'$, then $(G\odot H)\TII (G'\odot H')$.
\end{corollary}

As we already mentioned, the following has been proved in~\cite{LL.09}.

\begin{lemma}[\cite{LL.09}]
Let $G,H,F$ be graphs. Then, \[(G\TII H\ and\ H\TII F) \Rightarrow G\TII F.\]
\end{lemma}

Also, note that if $k\in {\cal F}(G)$, then there exists a surjective Type II homomorphism from $G$ to $K_k$. As a corollary, we get:

\begin{corollary}\label{cor:FallSpectrum}
Let $G,H$ be graphs and $\odot$ be an adjacency product. Then:
\[\bigcup_{\substack{k\in {\cal F}(G)\\k'\in {\cal F}(H)}}{\cal F}(K_k\odot K_{k'}) \subseteq \left(\bigcup_{k\in {\cal F}(H)}{\cal F}(G\odot K_k)\right) \cap \left(\bigcup_{k\in {\cal F}(G)}{\cal F}(K_k\odot H)\right),\]
and
\[\left(\bigcup_{k\in {\cal F}(H)}{\cal F}(G\odot K_k)\right) \cup \left(\bigcup_{k\in {\cal F}(G)}{\cal F}(K_k\odot H)\right) \subseteq {\cal F}(G\odot H).\]
\end{corollary}

Similarly to the previous section, we get the following.

\begin{lemma}\label{lem:fallPar}
Let $G,H$ be graphs and $\odot$ be an adjacency product, and let ${\cal F}$ denote the set $\bigcup_{p\in {\cal F}(G),q\in {\cal F}(H)}{\cal F}(K_p\odot K_q)$. If ${\cal F}(G)\neq$ and ${\cal F}(H)\neq\emptyset$, then $\emptyset\neq {\cal F} \subseteq {\cal F}(G\odot H)$, and:
\[\fa(G\odot H)\le \min\{k\mid k\in {\cal F}\}\le \fa(K_{\fa(G)}\odot K_{\fa(H)})\mbox{, and}\]

\[\fb(G\odot H)\ge \max\{k\mid k\in {\cal F}\} \ge \fb(K_{\fb(G)}\odot K_{\fb(H)}).\]
\end{lemma}



\section{Adjacency products of complete graphs}
\label{sec:products}

In this section we investigate the parameters of $K_p\odot K_q$ for the main adjacency products. The table below defines the condition $P_\odot$ for each of these products. If $uv\in E(G)$, we write $u\sim v$.

\begin{table}[htb]
\centering
\begin{tabular}{|c|c|c|}
\hline
{\bf Name} & {\bf Notation} & $P_\odot(G,H,(u_a,v_a),(u_b,v_b))$\\
\hline
					& & $u_a=u_b\ \wedge\ v_a\sim v_b$\\ 
Cartesian & $G\oblong H$ & $or$\\ & & $u_a\sim u_b\ \wedge\ v_a=v_b$\\
\hline
Direct & $G\times H$ & $u_a\sim u_b\ \wedge\ v_a\sim v_b$\\
\hline
					& & $u_a\sim u_b$\\
Lexicographic & $G[H]$ & $or$\\ & & $u_a=u_b\ \wedge\ v_a\sim v_b$\\
\hline
					& & $u_a=u_b\ \wedge\ v_a\sim v_b$\\
					& & $or$ \\
Strong & $G\boxtimes H$ & $u_a\sim u_b\ \wedge\ v_a=v_b$ \\
					& & $or$\\
					& & $u_a\sim u_b\ \wedge\ v_a\sim v_b$\\
\hline
Co-normal & $G*H$ & $u_a\sim u_b\vee v_a\sim v_b$ \\
\hline
\end{tabular}
\captionsetup{justification=centering}
\caption{Conditions for the existence of $(u_a,v_a)(u_b,v_b)$ in $E(G\odot H)$.}
\label{tab:products}
\end{table}

First, we investigate the structure of the product $K_p\odot K_q$. We write $G\cong H$ if $G$ and $H$ are isomorphic graphs.

\begin{proposition}\label{prop:Lex}
Let $p,q$ be positive integers and $\odot$ be either the lexicographic product, the strong product, or the co-normal product. Then \[K_p\odot K_q  \cong K_{pq}.\]
\end{proposition}
\begin{proof}
Write $V(K_p)$ as $\{u_1,\cdots,u_p\}$ and $V(K_q)$ as $\{v_1,\cdots,v_q\}$, and let $\alpha=(u_i,v_j)$ and $\beta=(u_h,v_k)$ be distinct vertices of $V(K_p\odot K_q)$. If $i\neq h$ and $j\neq k$, then one can verify that:
\begin{itemize}
\item $\alpha\beta\in E(K_p[K_q])$, since $u_i\sim u_h$;
\item $\alpha\beta\in E(K_p\boxtimes K_q)$, since $u_i\sim u_h$ and $v_j\sim v_k$;
\item $\alpha\beta\in E(K_p*K_q)$, since $u_i\sim u_h$;
\end{itemize} 
Now, suppose $j=k$, in which case $i\neq h$. We get the same situation as before for the lexicographic product and for the co-normal product. Also, $\alpha\beta\in E(K_p\boxtimes K_q)$, since $u_i\sim u_h$ and $v_j=v_k$. Finally, suppose that $i=h$ and $j\neq k$. Then, $\alpha\beta\in E(K_p[K_q])\cap E(K_p\boxtimes K_q)$, since $u_i=u_h$ and $v_j\sim v_k$, while $\alpha\beta\in E(K_p*K_q)$ since $v_j\sim v_k$.
\end{proof}

By the proposition above, we get that the value $pq$ is in the b-spectrum of $G\odot H$ for every $p\in S_b(G)$ and $q\in S_b(H)$, the same being valid for the fall-spectrum. This and Lemmas~\ref{lem:bGdotH} and~\ref{lem:fallPar} give us the corollary below and part of Theorems~\ref{theo:mainb} and~\ref{theo:mainfall}. We mention that $b(G\odot H)\ge b(G)b(H)$ has been proved in~\cite{JP.15}, when $\odot$ is either the lexicographic product or the strong product. Our result generalizes theirs, and we mention that, if more is learned about $b(G\odot K_p)$,  then Corollary~\ref{cor:bspectrum} can actually produce better bounds than the ones given in Theorems~\ref{theo:mainb} and~\ref{theo:mainfall}. 

\begin{corollary}\label{cor:Lex}
Let $G,H$ be graphs, and $\odot$ be the lexicographic, strong or co-normal product. Also, let $T$ denote either $S_b$ or ${\cal F}$. Then,
\[\{pq\mid p\in T(G),q\in T(H)\}\subseteq T(G\odot H).\]
\end{corollary}



Now, we analyse the colorings of the cartesian products. The following proposition will be useful.

\begin{proposition}\label{prop:Cart}
Let $p,q$ be positive integers. Then \[{\cal F}(K_p\oblong K_q) = \max\{p,q\}.\]
\end{proposition}
\begin{proof}
Consider $p\le q$, denote $K_p\oblong K_q$ by $G$, and write $V(K_p)$ as $\{u_1,\cdots,u_p\}$ and $V(K_q)$ as $\{v_1,\cdots,v_q\}$. First, we show that ${\cal F}(G)\neq \emptyset$. For this, let $f:V(G)\rightarrow \{1,\cdots,q\}$ be defined as follows: for every $j\in \{1,\cdots,q\}$, set $f((u_1,v_j))$ to $ j$; then color each subsequent $V(u_i)$ with a distinct chaotic permutation of $(1,\cdots,q)$ (there are enough permutations since $q\ge p$). Because every vertex is within a clique of size $q$, we get that $f$ is a fall-coloring. It remains to prove that no other fall-coloring exists. So, suppose by contradiction that $m\in {\cal F}(K_p\oblong K_q)\setminus\{q\}$, and let $f$ be a fall-coloring of $G$ with $m$ colors. Because $m>q$, there must exist a color $d$ that does not appear in $f(V(u_1))$. But since every vertex in $V(u_1)$ is a b-vertex and $V(u_i)$ is a clique for every $i\in \{1,\cdots,p\}$, this means that color $d$ must appear in $(u_{i_1},v_1),\cdots,(u_{i_q},v_q)$ for distinct values of $i_1,\cdots,i_q$, none of which can be~1. We get a contradiction since in this case we have $p\ge q+1$.
\end{proof}
We mention that the existence of a fall-coloring of $K_p\oblong K_q$ with $q$ colors has been first observed in~\cite{LL.09}, and that it already implies the part concerning the fall-spectrum in the lemma below. Nevertheless, Proposition~\ref{prop:Cart} tell us that, in order to get  bounds better than the one given by Theorem~\ref{theo:mainfall}, one needs to investigate ${\cal F}(G\oblong K_p)$ when $G$ is not the complete graph.

\begin{corollary}\label{cor:Cart}
Let $G,H$ be graphs. Then,
\[\{k\in S_b(G)\mid k\ge \chi(H)\}\cup \{k\in S_b(H)\mid k\ge \chi(G)\} \subseteq S_b(G\oblong H).\]
Also, if ${\cal F}\neq \emptyset$ and ${\cal F}(H)\neq$, then ${\cal F}(G\oblong H)\neq \emptyset$ and:
\[\{\max\{k,k'\}\mid k\in {\cal F}(G) \wedge k'\in {\cal F}(H)\}\subseteq {\cal F}(G\oblong H).\]
\end{corollary}
\begin{proof}
First, let $k\in S_b(G)$ be such that $k\ge \chi(H)$. By Proposition~\ref{prop:Cart}, we get $k\in S_b(K_k\oblong K_{\chi(H)})$. 
By Corollary~\ref{cor:bspectrum} and the fact that $\chi(H)\in S_b(H)$, we get that $k\in S_b(G\oblong H)$. Similarly, if $k\in S_b(H)$ is such that $k\ge \chi(G)$, we get that $k\in S_b(K_{\chi(G)}\oblong K_k)\subseteq S_b(G\oblong H)$.

Finally, let $k\in {\cal F}(G)$ and $k'\in {\cal F}(H)$. By Corollary~\ref{cor:FallSpectrum} and Proposition~\ref{prop:Cart}, we know that $\max\{k,k'\}\in {\cal F}(G\oblong H)$.
\end{proof}

In~\cite{KM.02}, the authors prove that $b(G\oblong H)\ge \max\{b(G),b(H)\}$. Observe that this also follows from the corollary above.

%

Regarding the direct product, in~\cite{JP.15} the authors observe that $b(G\times H)\ge \max\{b(G),b(H)\}$. Here, we give an alternate proof of this fact and show that when $G$ and $H$ are complete graphs, then there is equality. 
Our proof uses the following result.

\begin{theorem}\label{theo:KTV.02}\cite{KTV.02}
Let $G$ be isomorphic to the complete bipartite graph $K_{n,n}$ minus a perfect matching. Then $S_b(G) = \{2,n\}$.
\end{theorem}

Observe that the graph in the above theorem is isomorphic to the graph $K_2\times K_n$. In fact, if $G$ is the graph in the theorem above, one can easily verify that the 2-coloring and the $n$-coloring of $G$, which are unique, are also fall-colorings. Therefore, we also have ${\cal F}(G) = \{2,n\}$. This particular fact has been generalized in~\cite{DHHJKLR.00}, where the authors prove that ${\cal F}(K_p\times K_q) = \{p,q\}$. In the next theorem, we generalize both results by proving that in fact $S_b(K_p\times K_q)=\{p,q\}$. We mention that in~\cite{DHHJKLR.00} the authors observe that the theorem below cannot be generalized to the direct product of more than two complete graphs. In particular, they give a fall-coloring with~6 colors of $K_2\times K_3\times K_4$. 

\begin{theorem}\label{theo:Tensor}
Let $p, q$ be integers greater than~1. Then, 
\[{\cal F}(K_p\times K_q) = S_b(K_p\times K_q) = \{p,q\}\]
\end{theorem}
\begin{proof}
Write $V(K_p)$ as $\{u_1,\cdots,u_p\}$ and $V(K_q)$ as $\{v_1,\cdots,v_q\}$. For each $i\in\{1,\cdots,q\}$, denote by $C_i$ the set $V(v_i)$ (the vertices in fiber $v_i$), and for each $i\in \{1,\cdots,p\}$, denote by $R_i$ the set $V(u_i)$ (the vertices in fiber $u_i$). Denote $K_p\times K_q$ by $G$. Note that the coloring $f$ obtained by assigning color $i$ to every vertex in $C_i$, for every $i\in \{1,\cdots, q\}$, is a b-coloring of $G$ with $q$ colors; we say that $f$ is the \emph{column coloring of $G$}. We define the \emph{row coloring of $G$} analogously. Next, we prove that if $f$ is a b-coloring of $G$, then $f$ is either the column coloring or the row coloring of $G$. Because these colorings are also fall-colorings, the theorem follows.

We prove by induction on $q$. If $q=2$, we know it holds by Theorem~\ref{theo:KTV.02}; so suppose that $q\ge 3$. Because $K_p\times K_q$ is  isomorphic to $K_q\times K_q$, we can also suppose that $q\le p$.
Note that for every color $d$ used in $f$, there must exist $i$ such that $f^{-1}(d)\subseteq C_i$ or $f^{-1}(d)\subseteq R_i$. If the former occurs, we say that $d$ is a column color, and that it is a row color otherwise. 

First, suppose that there exists $i$ such that every vertex in $C_i$ have the same color, say $d$. Let $G_i = G -C_i$, and $f_i$ be equal to $f$ restricted to $G_i$. Note that, because $C_i = f^{-1}(d)$, we get that $f_i$ is a b-coloring of $G_i$, and by induction hypothesis, it is either the row or the column coloring of $G_i$. If $f_i$ is the column coloring, then we are done since it follows that $f$ is the column coloring of $G$. So, $f_i$ must be the row coloring, in which case we can suppose that for each $j\in \{1,\cdots,p\}$ we have that $f((u_j,v_\ell)) = j$ for every $\ell\in \{1,\cdots, q\}\setminus \{i\}$. But observe that, for each $j\in\{1,\cdots,p\}$, we have that $(u_j,v_i)$ misses color $j$, hence $(u_j,v_i)$ cannot be a b-vertex of color $d$. We get a contradiction since $f^{-1}(d) = C_i$.
Therefore, we can suppose that no column is monochromatic. 

Now, for each $i\in\{1,\cdots,p\}$, denote by $d_i$ the color $d$ such that $\lvert f^{-1}(d) \cap R_i\rvert >1$, if it exists. Denote by $R^*$ the set of row indices for which $d_i$ exists. Note that at most one such color exists per row as otherwise, if two colors are contained only in row $i$, then their vertices would be mutually non-adjacent and hence the colors would have no b-vertices. Similarly, each column $j$ contains at most one column color. 
For each $i\in R^*$, denote by $C_i$ the set $\{j\in \{1,\cdots,q\}\mid f((u_i,v_j)) = d_i\}$ (columns where $d_i$ appears). Finally, let $C^* = \bigcap_{i\in R^*}C_i$.  We first prove the following important facts:

\begin{enumerate}
\item\label{1} $R^*\neq \emptyset$: it follows because no column is monochromatic and no two column colors can be contained in the same column; 

\item\label{2} If $(u_i,v_j)$ is a b-vertex of color $d_i$, then $j\in C^*$: suppose otherwise, and let $i'\in R^*$ be such that $j\notin C_{i'}$; such index must exist since $j\notin C^*$. Let $j'\in C_{i'}$ be such that $(u_{i'},v_{j'})$ is a b-vertex of color $d_{i'}$; it must exist since $f^{-1}(d_{i'})\subseteq R_{i'}$. Finally, let $d=f((u_{i'},v_j))$. By the choice of $i'$, we know that $j'\neq j$, which implies $d\neq d_{i'}$ (recall that $j\notin C_i$). Furthermore, since $f^{-1}(d_i)\subseteq R_i$ and $i\neq i'$, we get $d\neq d_i$. Finally, since $R_{i'}$ can contain at most one row column, we get that $d$ is a column color. This implies that $(u_i,v_j)$ is not adjacent to color $d$, a contradiction.
\end{enumerate}

Now, without loss of generality, suppose that $R^*=\{1,\cdots, p'\}$ and that $C^*=\{1,\cdots,q'\}$. By~(1), we know that $q'\ge 1$. First, suppose that $p'<p$. By definition, we know that each color appears at most once in $R_i$ for every $i\in \{p'+1,\cdots, p\}$. This means that, for each $j\in C^*$, vertex $(u_1,v_j)$ is not adjacent to color $f((u_{p'+1},v_j))$, a contradiction since in this case, by~(2), color $d_1$ does not have b-vertices. Therefore, we have that $p'=p$. Now, suppose that $q'<q$. By the choice of $q'$, observe that there must exist a color $d\notin \{d_1,\cdots,d_p\}$ such that $f^{-1}(d)\subseteq C_{q'+1}$. Let $D = \{i\in\{1,\cdots,p\}\mid f((u_i,v_{q'+1})) = d\}$ (vertices in fiber $v_{q'+1}$ colored with $d$). Note that for each $i\in D$, we get that vertex $(u_i,v_{q'+1})$ is not adjacent to color $d_i$, a contradiction since in this case color $d$ has no b-vertices. Therefore, we have that $q'=q$, in which case $f$ is the row b-coloring of $G$.
\end{proof}


\begin{corollary}\label{cor:Direct}
Let $G,H$ be graphs. Then, 
\[S_b(G)\cup S_b(H) \subseteq S_b(G\times H)\]
Furthermore, if ${\cal F}(G)\neq \emptyset$ and ${\cal F}(H)\neq \emptyset$, then
\[{\cal F}(G)\cup {\cal F}(H) \subseteq {\cal F}(G\times H).\]
\end{corollary}
\begin{proof}
By Corollary~\ref{cor:bspectrum}, we know that for every $p\in S_b(G)$ and $q\in S_b(H)$, we have that $S_b(K_p\times K_q)$ is contained in $S_b(G\times H)$. And by Theorem~\ref{theo:Tensor} and the fact that $S_b(F)\neq \emptyset$, for every graph $F$, we get that 
\[\bigcup_{\substack{p\in S_b(G)\\q\in S_b(H)}}S_b(K_p\times K_q) = \bigcup_{\substack{p\in S_b(G)\\q\in S_b(H)}}\{p,q\} = S_b(G)\cup S_b(H).\]

By Corollary~\ref{cor:FallSpectrum}, the same argument can be applied for fall-colorings as long as the product $K_p\times K_q$ is defined for some value of $p$ and some value of $q$, i.e., as long as ${\cal F}(G)\neq \emptyset$ and ${\cal F}(H)\neq \emptyset$.
\end{proof}

Observe that the first part of Theorems~\ref{theo:mainb} and~\ref{theo:mainfall} are given by Corollary~\ref{cor:Lex}, while the cartesian product and the direct product parts are given by Corollaries~\ref{cor:Cart} and~\ref{cor:Direct}, respectively.


\section{Fall coloring and products of general graphs}\label{sec:fallPart}

We have seen that ${\cal F}(G\odot H)\neq \emptyset$ whenever ${\cal F}(G)\neq \emptyset$ and ${\cal F}(H)\neq \emptyset$, but what happens when one of these sets is empty? Next, we show some situations where we still can obtain a fall-coloring of the product, even though one of the fall-spectra might be empty.

\begin{theorem}
Let $G,H$ be graphs, and suppose that ${\cal F}(G)\neq \emptyset$. Then 
\[\{k\in {\cal F}(G)\mid k\ge \chi(H)\}\subseteq {\cal F}(G\oblong H).\]
\end{theorem}
\begin{proof}
Let $f$ be any fall-coloring of $G$ that uses colors $\{1,\cdots,k\}$, where $k\ge \chi(H)$, and consider an optimal coloring $g$ of $H$ that uses colors $\{1,\cdots,\ell\}$. Then, for each $i\in\{1,\cdots,\ell\}$, let $\pi_i$ denote the permutation \[(i,i+1,\cdots,k,1,\cdots,i-1).\] Note that $\pi_i$ is well defined since $k\ge \ell$. Finally, for each $u\in g^{-1}(i)$,  color the copy of $G$ related to $u$ by using $f$ where the colors are permuted as in $\pi_i$. Let $h$ be the obtained coloring. 
Because $h$ restricted to each copy of $G$ is nothing more than a permutation of the colors used in $f$, we get that every vertex is still a b-vertex. 
\end{proof}

\begin{theorem}
Let $G,H$ be graphs and suppose that ${\cal F}(G)\neq \emptyset$ and $\chi(G)>1$. Then ${\cal F}(G)\subseteq {\cal F}(G\times H)$ if and only if $H$ has no isolated vertices.
\end{theorem}
\begin{proof}
If $H$ has isolated vertices, so does $G\times H$, and since $\chi(G)\ge 2$, these vertices can never be b-vertices; hence ${\cal F}(G\times H) = \emptyset$. Now, let $f$ be a fall-coloring of $G$ with $k$ colors, and let $g:V(G\times H)\rightarrow \{1,\cdots,k\}$ be defined  as $g((u,v)) = f(u)$, for every $(u,v)\in V(G\times H)$. Consider a color $i$; because $f^{-1}(i)$ is a stable set, as well as  $V(u)$ for every $u\in f^{-1}(i)$, we get that $g^{-1}(i)$ is also a stable set (i.e., $g$ is a proper coloring). Now, consider a vertex $(u,v)\in V(G\times H)$. Since $H$ has no isolated vertices, we get that $v$ must have some neighbor, say $v'$. By definition, we know that $S = \{(u',v')\in V(G\times H)\mid u'\in N(u)\}$ is contained in $N((u,v))$, and because $u$ is a b-vertex in $f$, we know that $f(N(u)) = g(S) = \{1,\cdots,k\}\setminus \{f(u)\}$. It follows that $(u,v)$ is a b-vertex in $g$.
\end{proof}

\section{Conclusion}\label{sec:conclusion}



We have seen that the spectrum of products involving complete graphs play an important role in better understanding the spectrum of general graphs. Then, we set out to investigate the main graph products and have seen that: ${\cal F}(K_p\boxtimes K_q) = S_b(K_p\boxtimes K_q) = \{pq\}$ (Proposition~\ref{prop:Lex});  ${\cal F}(K_p\oblong K_q) = \max\{p,q\}$ (Proposition~\ref{prop:Cart}); and ${\cal F}(K_p\times K_q) = S_b(K_p\times K_q) = \{p,q\}$ (Theorem~\ref{theo:Tensor}). Therefore, the only not completely described set is $S_b(K_p\oblong K_q)$.
This however seems to be a much harder problem, as hinted by the results presented in~\cite{JO.12}. There, the authors show that $b(K_n\oblong K_n)\ge 2n-3$ and they conjecture that this is best possible. However, their conjecture does not hold, as can be seen in~\cite{MB.15}. Nonetheless, following their result, we pose the question below.

\begin{question}
Let $p$ and $q$ be positive integers. Does the following hold? \[b(K_p\oblong K_q) \ge p+q-3.\]
\end{question}

We mention that in~\cite{KM.02}, the authors prove that $b(G\oblong H)\ge b(G)+b(H)-1$ under certain conditions. Note that if the answer to the above question is ``yes'', then Corollary~\ref{cor:bspectrum} implies $b(G\oblong H) \ge b(G)+b(H)-3$. This would considerably improve previous results, since the conditions for $b(G\oblong H)\ge b(G)+b(H)-1$ in~\cite{KM.02} are quite strong.

Concerning the existence of fall-colorings, we have seen that ${\cal F}(G\odot H)\neq \emptyset$ whenever ${\cal F}(G)\neq \emptyset$ and ${\cal F}(H)\neq \emptyset$. We also have seen that under some conditions ${\cal F}(G\oblong H)\neq \emptyset$ and ${\cal F}(G\times H)\neq \emptyset$ when ${\cal F}(G)\neq \emptyset$. 
In~\cite{S.15_2}, the authors observe that the graph $C_5[K_2]$ has a fall-coloring, while we know that $C_5$ has no fall-colorings. Because of the next proposition, we get that $C_5[K_2] \cong C_5\boxtimes K_2$. 

\begin{proposition}
Let $G$ and $H$ be graphs. Then, $G\boxtimes H\subseteq G[H]$, with equality when $H$ is the complete graph.
\end{proposition}
\begin{proof}
First, let $e=(u_a,v_a)(u_b,v_b)\in E(G\boxtimes H)$. If the first condition of the definition of strong product holds, we trivially get that $e\in G[H]$; therefore one of the other two conditions hold, i.e., we have $u_au_b\in E(G)$ and, again, we get $e\in G[H]$.

Now, suppose that $H$ is the complete graph and let $e=(u_a,v_a)(u_b,v_b)\in E(G[H])$. As before, if the second condition of the definition of the lexicographic product is also one of the conditions in the strong product. So, we can consider that $u_au_b\in E(G)$. In this case, either $v_a=v_b$ in which case $e$ is also in $G\boxtimes H$, or $v_a\neq v_b$ in which case $e$ is in $G\boxtimes K_p$ since $H$ is complete.
\end{proof}

Therefore, for all of these products there are cases where ${\cal F}(G\odot H)\neq \emptyset$ even thought one of ${\cal F}(G)$ and ${\cal F}(H)$ is empty. Therefore, a good question is what happens when both these sets is empty. 
In~\cite{S.15_2}, they prove that if ${\cal F}(H) = \emptyset$, then ${\cal F}(G[H])=\emptyset$. This means that the answer to the following question is ``no'' for the lexicographic product.

\begin{question}\label{question:emptyset}
Let $\odot$ be an adjacency product. Does there exist graphs $G$ and $H$ such that ${\cal F}(G)=\emptyset$, ${\cal F}(H)=\emptyset$, and ${\cal F}(G\odot H)\neq \emptyset$? 
\end{question}

%
%

Finally, we present an example that answers in the negative the following question, posed by Kaul and Mitillos. 

\begin{question}\label{conj:KM}\cite{KM.16}
Does the following hold whenever $G$ is a perfect graph?
\[\chi(G)=\delta(G)+1 \Leftrightarrow {\cal F}(G) = \{\chi(G)\}.\]
\end{question}

Observe that if $G$ is a chordal graph, then $\omega(G)\ge \delta(G)+1$. Also, recall that if ${\cal F}(G)\neq \emptyset$, then $\chi(G)\le \fa(G)\le \delta(G)+1$. Therefore, we know that if $G$ is chordal and ${\cal F}(G)\neq \emptyset$ then ${\cal F}(G)=\{\chi(G)\}=\{\delta(G)+1\}$, i.e., the necessary part of the question holds for chordal graphs. However, as we show in the next paragraph, the sufficient part does not always hold for chordal graphs.

Let $G_1$ be constructed as follows: start with a path $P=(v_1,\cdots,v_6)$ of size~6; add a vertex $u$ and edges between $u$ and each $v_i$ in $P$; add a peding clique of size~6 adjacent to $v_i$, for every $i\in \{1,2,5,6\}$. Now, let $G_2$ be obtained as follows: start with a clique $C=\{u_1,u_2,u_3,u_4\}$ of size~4; add vertices $x$ and $y$ and edges $\{xy,xu_1,xu_2,yu_1,yu_2\}$; add $z_1$ adjacent to $u_1$ and $u_3$; add $z_2$ adjacent to $u_2$ and $u_4$; then, for every vertex $v\in\{z_1,z_2,x,y\}$, add a pending clique of size~6 adjacent to $v$. Finally, let $G$ be obtained from $G_1$ and $G_2$ by glueing the edges $v_3v_4$ and $u_3u_4$. 
It is not hard to see that $G$ is a chordal graph, since it can be obtained from cliques by glueing them along an edge or along a vertex. Observe that $\delta(G)= 6$ and that $\omega(G)=7$; hence $\chi(G)=\delta(G)+1$. We show that ${\cal F}(G)=\emptyset$. Let $c$ be any optimal coloring of $G$, and suppose that $u,u_1,u_2$ are b-vertices in $c$. We prove that $v_3,v_4$ cannot be both b-vertices; this implies our claim. Note that $u,u_1,u_2$ all have degree exactly~6, which means that every vertex in their neighborhoods must have distinct colors. Therefore, we get $c(v_2)\neq c(v_5)$, and since $N(u_1)\setminus N[u_2] = \{z_1\}$ and $N(u_2)\setminus N[u_1] = \{z_2\}$, we get $c(z_1) = c(z_2)$. Denote by $i$ the color of $z_1$. But now, $\{c(v_2),i\} = c(N(v_3))\setminus c(N[v_4]) \neq c(N(v_4))\setminus c(N[v_3]) = \{c(v_5),i\} $, which cannot hold when $v_3$ and $v_4$ are both b-vertices. 

Nevertheless, one can still ask about the maximal subclasses of the perfect graphs for which the answer is ``yes''. For instance, it has been proved to hold for threshold graphs and split graphs~\cite{KM.16}, and for strongly chordal graphs~\cite{LDL.05} In particular, we pose the following question.

\begin{question}
Can one decide in polynomial time whether a chordal graph $G$ is such that ${\cal F}(G)\neq \emptyset$?
\end{question}


\bibliographystyle{plain}
\bibliography{../refs}
\end{document}